\documentclass[11pt]{amsart}%
\usepackage{graphicx,amscd,color,amsmath,amsfonts,amssymb,geometry}
\usepackage{amsmath}
\usepackage{amsfonts}
\usepackage{amssymb}
\usepackage{graphicx}%
\providecommand{\U}[1]{\protect\rule{.1in}{.1in}}
\newtheorem{theorem}{Theorem}[section]
\newtheorem{remark}[theorem]{Remark}

\newtheorem{proposition}[theorem]{Proposition}

\begin{document}
\title[New lower bounds for the constants in the Hardy--Littlewood inequality]{New lower bounds for the constants in the real polynomial Hardy--Littlewood inequality}
\author[W. Cavalcante]{W. Cavalcante}
\address{Departamento de Matem\'{a}tica, \\
\indent Universidade Federal de Pernambuco, \\
\indent50.740-560 - Recife, Brazil.}
\email{wasthenny.wc@gmail.com}
\author[D. N\'{u}\~{n}ez]{D. N\'{u}\~{n}ez-Alarc\'{o}n}
\address{Departamento de Matem\'{a}tica, \\
\indent Universidade Federal de Pernambuco, \\
\indent50.740-560 - Recife, Brazil.}
\email{danielnunezal@gmail.com}
\author[D. Pellegrino]{D. Pellegrino}
\address{Departamento de Matem\'{a}tica, \\
\indent Universidade Federal da Para\'{\i}ba, \\
\indent 58.051-900 - Jo\~{a}o Pessoa, Brazil.}
\email{pellegrino@pq.cnpq.br and dmpellegrino@gmail.com}
\thanks{D. Pellegrino was supported by CNPq Grant 477124/2012-7 and INCT-Matem\'{a}tica.}

\begin{abstract}
In this short note we obtain new lower bounds for the constants of the real
Hardy--Littlewood inequality for $m$-linear forms on $\ell_{p}^{2}$ spaces
when $p=2m$ and for certain values of $m$. The real and complex cases for the
general case $\ell_{p}^{n}$ were recently investigated in \cite{ara} and
\cite{sete}. When $n=2$ our results improve the best known estimates for these constants.

\end{abstract}
\maketitle

\section{Introduction}

The Hardy--Littlewood inequality for multilinear forms and homogeneous
polynomials in $\ell_{p}$ spaces dates back to 1934 \cite{hardy} for the
bilinear case, as a beautiful and highly nontrivial optimal extension of
Littlewood's $4/3$ inequality from $\ell_{\infty}^{n}$ to $\ell_{p}^{n}$
spaces. In 1980 Praciano-Pereira \cite{pra} extended the Hardy--Littlewood
inequality to $m$-linear operators for $p\geq2m$ and recently, in 2013, Dimant
and Sevilla-Peris \cite{dimant} obtained an optimal extension for the case
$m<p<2m.$ Both the multilinear and polynomial cases of this inequality were
deeply investigated in recent years and perhaps the main motivation is the
fact that when $p=\infty$ we recover the classical Bohnenblust--Hille
inequality \cite{bh} from 1931, which has found, since 2011, new striking
applications in many fields of Mathematics and even in Quantum Information
Theory (see, for instance, \cite{bps, ann, monta} and the references therein).

\bigskip Henceforth, for any map $f:\mathbb{R}\rightarrow\mathbb{R}$ we
define
\[
f\left(  \infty\right)  :=\lim_{p\rightarrow\infty}f(p).
\]

\bigskip For $\mathbb{K}$ be $\mathbb{R}$ or $\mathbb{C}$ and $\alpha
=(\alpha_{1},\ldots,\alpha_{n})\in{\mathbb{N}}^{n}$, we define $|\alpha
|:=\alpha_{1}+\cdots+\alpha_{n}$. By $\mathbf{x}^{\alpha}$ we shall mean the
monomial $x_{1}^{\alpha_{1}}\cdots x_{n}^{\alpha_{n}}$ for $\mathbf{x}%
=(x_{1},\ldots,x_{n})\in{\mathbb{K}}^{n}$. The polynomial Bohnenblust--Hille
inequality (see \cite{bh}, 1931) asserts that, given $m,n\geq1$, there is a
constant $B_{\mathbb{K},m}^{\mathrm{pol}}\geq1$ such that
\[
\left(  {\sum\limits_{\left\vert \alpha\right\vert =m}}\left\vert a_{\alpha
}\right\vert ^{\frac{2m}{m+1}}\right)  ^{\frac{m+1}{2m}}\leq B_{\mathbb{K}%
,m}^{\mathrm{pol}}\left\Vert P\right\Vert
\]
for all $m$-homogeneous polynomials $P:$ $\ell_{\infty}^{n}\rightarrow
\mathbb{K}$ given by
\[
P(x_{1},...,x_{n})=\sum_{|\alpha|=m}a_{\alpha}\mathbf{{x}^{\alpha},}%
\]
and all positive integers $n$, where $\Vert P\Vert:=\sup_{z\in B_{\ell
_{\infty}^{n}}}|P(z)|$. It is well-known that the exponent $\frac{2m}{m+1}$ is sharp.

When one tries to replace $\ell_{\infty}^{n}$ by $\ell_{p}^{n}$ the extension
of the polynomial Bohnenblust--Hille inequality is called polynomial
Hardy--Littlewood inequality and the optimal exponents are $\frac
{2mp}{mp+p-2m}$ for $2m\leq p\leq\infty$ and $\frac{p}{p-m}$ for $m<p<2m$.
More precisely, given $m,n\geq1$, there is a constant $C_{\mathbb{K}%
,m,p}^{\mathrm{pol}}\geq1$ such that
\[
\left(  {\sum\limits_{\left\vert \alpha\right\vert =m}}\left\vert a_{\alpha
}\right\vert ^{\frac{2mp}{mp+p-2m}}\right)  ^{\frac{mp+p-2m}{2mp}}\leq
C_{\mathbb{K},m,p}^{\mathrm{pol}}\left\Vert P\right\Vert ,
\]
for all $m$-homogeneous polynomials on $\ell_{p}$ with $2m\leq p\leq\infty$
given by $P(x_{1},\ldots,x_{n})=\sum_{|\alpha|=m}a_{\alpha}\mathbf{{x}%
^{\alpha}}$. When $m<p<2m$ the optimal exponent is $\frac{p}{p-m}.$

The search for precise estimates in the polynomial and multilinear
Bohnenblust--Hille inequalities has been pursued by many authors (\cite{bps,
ann, diana, ala, npss} and the references therein) and is important for many
different reasons besides its intrinsic mathematical challenge. The knowledge
of precise estimates for the constants of the Bohnenblust--Hille inequalities
is a crucial point for applications (see \cite{bps, ann, monta}). In this
paper we improve the best known lower bounds for the constants of the
polynomial Hardy--Littlewood inequality for the case of real scalars for
certain values of $m$. In some sense, there is a big difference between the
Hardy--Littlewood and Bohnenblust--Hille inequalities. While in the
Bohnenblust--Hille inequality the domain $\ell_{\infty}^{n}$ remains
unchanged, in the Hardy--Littlewood inequality the variable $p$ in $\ell
_{p}^{n}$ depends on the degree of multilinearity. For this reason the
expression \textquotedblleft asymptotic growth\textquotedblright\ makes sense
for the constants of the Bohnenblust--Hille inequality but needs much care in
the case of the Hardy--Littlewood inequality. In this paper, we choose the
case $p=2m$, which seems to be a distinguished case (see comments in
\cite{ara2}) but similar investigation can be done for the other cases.

\bigskip

\section{Lower estimates of constants on the Hardy--Littlewood inequality}

\bigskip In this section we use polynomials introduced in 2013 by J.R. Campos
et al. \cite{campos1} and also in 2015 by P. Jimenez et al. \cite{GUSE} in the
case of the Bohnenblust--Hille inequalities:%

\begin{equation}%
\begin{tabular}
[c]{lll}%
$P_{2}\left(  x,y\right)  $ & $=$ & $\pm\left(  ax^{2}-ay^{2}\pm
2\sqrt{a\left(  1-a\right)  }xy\right)  $\\
$P_{3}\left(  x,y\right)  $ & $=$ & $ax^{3}+bx^{2}y+bxy^{2}+ay^{3}$\\
$P_{5}\left(  x,y\right)  $ & $=$ & $ax^{5}-bx^{4}y-cx^{3}y^{2}+cx^{2}%
y^{3}+bxy^{4}-ay^{5}$\\
$P_{6}\left(  x,y\right)  $ & $=$ & $ax^{5}y+bx^{3}y^{3}+axy^{5}$\\
$P_{7}\left(  x,y\right)  $ & $=$ & $-ax^{7}+bx^{6}y+cx^{5}y^{2}-dx^{4}%
y^{3}-dx^{3}y^{4}+cx^{2}y^{5}+bxy^{6}-ay^{7}$\\
$P_{8}\left(  x,y\right)  $ & $=$ & $-ax^{7}y+bx^{5}y^{3}-bx^{3}y^{5}+axy^{7}%
$\\
$P_{10}\left(  x,y\right)  $ & $=$ & $ax^{9}y+bx^{7}y^{3}+x^{5}y^{5}%
+bx^{3}y^{7}+axy^{9}.$%
\end{tabular}
\ \ \ \label{polis}%
\end{equation}
When dealing with $p=\infty$, the domain of these polynomials is always
$l_{\infty}^{2}\left(  \mathbb{R}\right)  $ and, in each case, it was
investigated in \cite{campos1, GUSE} the best choice of the parameters
$a,b,c,d$ in such a way that we obtain good (er even best) lower bounds for
the Bohnenblust--Hille inequality when the domain is $l_{\infty}^{2}\left(
\mathbb{R}\right)  $.

At a first glance we shall use the same polynomials from \cite{campos1, GUSE}
(that is, with the same parameters $a,b,c,d$ from \cite{campos1, GUSE}) to
estimate the constants $C_{\mathbb{R},m,2m}^{\mathrm{pol}}$ of the real
polynomial Hardy--Littlewood inequality when $p=2m$. For this task we shall
estimate the polynomial norms and recall that now the domain is $l_{2m}%
^{2}\left(  \mathbb{R}\right)  $. Estimating by analytical means the norms
$\Vert P\Vert:=\sup_{z\in B_{\ell_{p}^{2}}}|P(z)|$ when $2m\leq p<\infty$
seems to be not possible in general (or, at least, highly nontrivial) and for
this task we shall make a computer-assisted approach. The computational
procedure uses the software Matlab with the interior point algorithm, by
discretizing the region to find the a good initial point (an initial point to
search the maximum); after that we use an optimization algorithm and a global
search method to obtain the maximum. We recall that the parameters of the
above polynomials in \cite{campos1, GUSE} are:%
\[%
\begin{tabular}
[c]{lll}%
$P_{2}\left(  x,y\right)  $ & $\rightarrow$ & $a=$ $0.867835$\\
$P_{3}\left(  x,y\right)  $ & $\rightarrow$ & $a=1,~b=-1.6692$\\
$P_{5}\left(  x,y\right)  $ & $\rightarrow$ &
$a=0.19462,~b=0.66008,~c=0.97833$\\
$P_{6}\left(  x,y\right)  $ & $\rightarrow$ & $a=1,~b=-2.2654$\\
$P_{7}\left(  x,y\right)  $ & $\rightarrow$ &
$a=0.05126,~b=0.22070,~c=0.50537,~d=0.71044$\\
$P_{8}\left(  x,y\right)  $ & $\rightarrow$ & $a=0.15258,~b=0.64697$\\
$P_{10}\left(  x,y\right)  $ & $\rightarrow$ & $a=0.0938,~b=-0.5938.$%
\end{tabular}
\ \ \
\]
The following table shows the estimates obtained for $\Vert P_{m}\Vert
=\sup_{z\in B_{\ell_{2m}^{2}}}|P(z)|$ and $C_{\mathbb{R},m,2m}^{\mathrm{pol}}%
$:%
\[%
\begin{tabular}
[c]{lll}%
Polynomial $P_{m}$ & Norm of $P_{m}$ in $B_{\ell_{2m}^{2}}$ & Lower estimate
for $C_{\mathbb{R},m,2m}^{\mathrm{pol}}$\\
$P_{2}$ & $\approx0.991227730027263$ & $\geq1.414213562373095>\left(
1.18\right)  ^{2}$\\
$P_{3}$ & $\approx1.336725475130557$ & $\geq2.058620016006847>\left(
1.27\right)  ^{3}$\\
$P_{5}$ & $\approx0.286160496407654$ & $\geq5.911278874557850>\left(
1.42\right)  ^{5}$\\
$P_{6}$ & $\approx0.265449175431079$ & $\geq10.06063557813303>\left(
1.46\right)  ^{6}$\\
$P_{7}$ & $\approx0.071365688615534$ & $\geq17.850856996050050>\left(
1.50\right)  ^{7}$\\
$P_{8}$ & $\approx0.029851212141614$ & $\geq31.491320225749660>\left(
1.53\right)  ^{9}$\\
$P_{10}$ & $\approx0.015289940437748$ & $\geq85.844178992096431>\left(
1.56\right)  ^{10}$%
\end{tabular}
\ \ \ \ \ \
\]

\bigskip These estimates are better than the best known estimates from
\cite{sete}. In the next section we obtain even better estimates.

\begin{remark}
\bigskip An interesting point that must be stressed is that, contrary to what
is done in the case $p=\infty$ (the Bohnenblust--Hille inequality) in
\cite{campos1, GUSE}, when good estimates for the case $km$-homogeneous for
$k>1$ are obtained by using the polynomials for the case $m$-homogeneous just
by multiplying the respective polynomials, in the case of the
Hardy--Littlewood inequality this procedure is not possible because the domain
of the $m$-homogeneous polynomials depend on $m$, i.e., in general $p>m,$ and
in the case studied here $p=2m$, so it is obviously not possible to proceed as
in the case of the Bohnenblust--Hille inequality.
\end{remark}

\section{Lower estimates of constants on the Hardy--Littlewood inequality:
finding new polynomials and better estimates}

The purpose of this section is to find the best parameters $a,b,c,d$ for the
polynomials (\ref{polis}) in such a way that we obtain lower bounds for
$C_{\mathbb{R},m,2m}^{\mathrm{pol}}$ better than those of the previous
section. Here the degree of complexity increases since we would have to
formally test infinitely many possibilities for the parameters $a,b,c,d$ to
maximize the quotient $\frac{\left\vert P_{m}\right\vert _{2}}{\left\Vert
P_{m}\right\Vert }$. Besides, contrary to the case of the Bohnenblust--Hille
theory, there seems to be (as far as we know) no theory of extremal
polynomials developed for this case. It is well-known that the maxima of
homogeneous polynomials in the unit ball is achieved in the unit sphere and so
we shall work on the unit sphere; this helps in the computational estimates.
We also parametrize $y$ as a function of $x,$ and this also helps to make
calculations faster.

We have numerical evidence that the best constant $C_{\mathbb{R},2,4}$ when restricted to two variables is given
by the polynomial of the polynomial $P_{2}$ of the previous section.

The estimates obtained in this section can be summarized as follows:

\begin{proposition}
The constants of the real polynomial Hardy--Littlewood inequality satisfy:%

\begin{align*}
C_{\mathbb{R},2,4}^{\mathrm{pol}} &  \geq1.414213562373095\approx\sqrt{2}\\
C_{\mathbb{R},3,6}^{\mathrm{pol}} &  \geq2.236067>\left(  1.30\right)  ^{3}\\
C_{\mathbb{R},5,10}^{\mathrm{pol}} &  \geq6.191704>\left(  1.44\right)  ^{5}\\
C_{\mathbb{R},6,12}^{\mathrm{pol}} &  \geq10.636287>\left(  1.48\right)
^{6}\\
C_{\mathbb{R},7,14}^{\mathrm{pol}} &  \geq18.095148>\left(  1.51\right)
^{7}\\
C_{\mathbb{R},8,16}^{\mathrm{pol}} &  \geq31.727174>\left(  1.54\right)
^{8}\\
C_{\mathbb{R},10,20}^{\mathrm{pol}} &  \geq91.640152>\left(  1.57\right)
^{10}%
\end{align*}

\end{proposition}

\begin{proof}
The parameters that we found for $a,b,c,d$ are the following, which improve
the estimates of the previous section we note that the estimates are (we note
that when trying to find the coefficients we have numerical evidence that the
best constant for the case $m=2$ is given by the polynomial of the polynomial
$P_{2}$ of the previous section:%

\[%
\begin{tabular}
[c]{lll}%
$P_{3}\left(  x,y\right)  $ & $\rightarrow$ & $a=1,~b\approx-2$\\
$P_{5}\left(  x,y\right)  $ & $\rightarrow$ & $a\approx0.104245,~b\approx
0.333366,~c\approx0.541712$\\
$P_{6}\left(  x,y\right)  $ & $\rightarrow$ & $a=1,~b\approx-2.363681$\\
$P_{7}\left(  x,y\right)  $ & $\rightarrow$ & $a\approx0.0555555,~b\approx
0.2444444,~c\approx0.5555555,~d\approx0.8000000$\\
$P_{8}\left(  x,y\right)  $ & $\rightarrow$ & $a\approx0.210344,~b\approx
0.896551$\\
$P_{10}\left(  x,y\right)  $ & $\rightarrow$ & $a\approx0.085714,~b\approx
-0.577551.$%
\end{tabular}
\ \ \
\]

\begin{figure}[tbh]
\centering
\includegraphics[scale = .7]{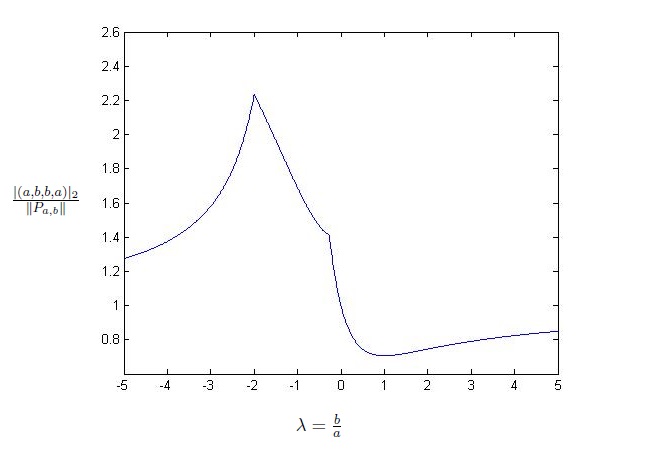}\caption{Graph of the quotient
$\frac{|(a,b,b,a)|_{2}}{\|P_{a,b} \|}$ as function of $\lambda= \frac{b}{a}$ -
Polynomial $P_{3}$}%
\end{figure}

\begin{figure}[tbh]
\centering
\includegraphics[scale = .7]{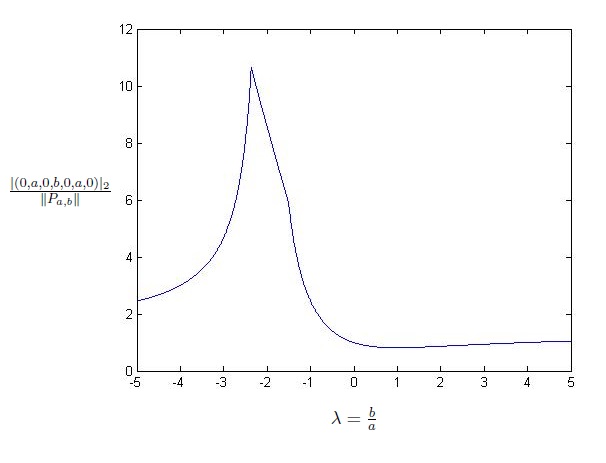}\caption{Graph of the quotient
$\frac{|(0,a,0,b,0,a,0)|_{2}}{\|P_{a,b} \|}$ as function of $\lambda= \frac
{b}{a}$ - Polynomial $P_{6}$}%
\end{figure}

The following table shows the estimates obtained (the estimates for $P_{2}$
are the same of the previous section):%
\[%
\begin{tabular}
[c]{lll}%
Polynomial $P_{m}$ & Norm of $P_{m}$ in $B_{\ell_{2m}^{2}}$ & Lower estimate
for $C_{\mathbb{R},m,2m}^{\mathrm{pol}}$\\
$P_{3}$ & $\approx1.414213$ & $\geq2.236067$\\
$P_{5}$ & $\approx0.147219$ & $\geq6.191704$\\
$P_{6}$ & $\approx0.258967$ & $\geq10.636287$\\
$P_{7}$ & $\approx0.078601$ & $\geq18.095148$\\
$P_{8}$ & $\approx0.041048$ & $\geq31.727174$\\
$P_{10}$ & $\approx0.014151$ & $\geq91.640152$%
\end{tabular}
\ \
\]

\end{proof}

\section{Asymptotic hypercontractivity constant $H_{R,p}\left(  n\right)  $}

In \cite{campos1, GUSE} it is defined:%
\[
C_{\mathbb{K},m,\infty}^{\mathrm{pol}}\left(  n\right)  :=\inf\left\{
C>0:\left\vert P\right\vert _{\frac{2m}{m+1}}\leq C\left\Vert P\right\Vert
\text{, for all }P\in\mathcal{P}\left(  ^{m}\ell_{\infty}^{n}\left(
\mathbb{K}\right)  \right)  \right\}
\]
and%
\[
H_{\mathbb{K},\infty}:=\underset{m}{\lim\sup}\sqrt[m]{C_{\mathbb{K},m,\infty
}^{\mathrm{pol}}}%
\]
and
\[
H_{\mathbb{K},\infty}\left(  n\right)  :=\underset{m}{\lim\sup}\sqrt[m]%
{C_{\mathbb{K},m,\infty}^{\mathrm{pol}}\left(  n\right)  }.%
\]
Analogously, we can define for $2m\leq p\leq\infty$:%
\[
C_{\mathbb{K},m,p}^{\mathrm{pol}}\left(  n\right)  :=\inf\left\{
C>0:\left\vert P\right\vert _{\frac{2mp}{mp+p-2m}}\leq C\left\Vert
P\right\Vert \text{, for all }P\in\mathcal{P}\left(  ^{m}\ell_{p}^{n}\left(
\mathbb{K}\right)  \right)  \right\}
\]
and
\[
H_{\mathbb{K},p}:=\underset{m}{\lim\sup}\sqrt[m]{C_{\mathbb{K},m,p}%
^{\mathrm{pol}}}%
\]
and
\[
H_{\mathbb{K},p}\left(  n\right)  :=\underset{m}{\lim\sup}\sqrt[m]%
{C_{\mathbb{K},m,p}^{\mathrm{pol}}\left(  n\right)  }%
\]
For the real case, it has been recently proved in \cite{campos2} that
$H_{\mathbb{R},p}$ $\geq2$. However, not many exact values of $H_{\mathbb{R}%
,p}\left(  n\right)  $ are known so far. In \cite{campos1, GUSE} it is studied $H_{\mathbb{R},\infty}\left(  2\right)  $. In fact, in those works it is concluded that
\[
H_{\mathbb{R},\infty}\left(  2\right)  \geq1.65171.
\]
In the present section we shall estimate $H_{\mathbb{R},2m}\left(  2\right)  $. Here, the analytical estimates have shown even more difficult, since the domains in which we work do not keep the same, contrary to what happens when working in $B_{\ell_{\infty}^{n}}$. For this reason, the use of new computational techniques to estimate $H_{\mathbb{R}%
,2m}\left(  2\right)  $ are needed. Our computational conclusion is that 
\[
H_{\mathbb{R},2m}\left(  2\right)  \geq1.65362.
\]

In all estimates we have used $m=600$, but besides we want to estimate  $H_{\mathbb{R},2m}\left(  2\right)  $ (working as in \cite{campos1, GUSE}), using the polynomials defined in 
(\ref{polis}). A first attempt in this direction would be to use the new parameters introduced in the precious section.

\bigskip In the following tables we use the polynomials from 
(\ref{polis}).
\[%
\begin{tabular}
[c]{lll}%
Polynomial $P_{600}$ & New parameters of the previous section &
Lower estimate for $H_{\mathbb{R},1200}\left(  2\right)  $\\
$(P_{3})^{200}$ & $a=1,~b\approx-2$ & $\geq1.288250$\\
$(P_{5})^{120}$ & $a\approx0.104245,~b\approx0.333366,~c\approx0.541712$ &
$\geq1.457854$\\
$(P_{6})^{100}$ & $a=1,~b\approx-2.363681$ & $\geq1.509926$\\
$(P_{8})^{75}$ & $a\approx0.191919,~b\approx0.8181818$ & $\geq1.637228$\\
$(P_{10})^{60}$ & $a\approx0.085714,~b\approx-0.577551.$ & $\geq1.638615$%
\end{tabular}
\ \ \ \ \ \ \
\]
Following the spirit of the previous section, it makes sense to think that as we are working in a new domain, different from the one where our parameters seem effective, it may exist better parameters furnishing better estimates for  $H_{\mathbb{R}%
,2m}\left(  2\right)  $. We also observe that when $m$ is big, in some sense the ball 
$B_{\ell_{2m}^{n}}$ is close to $B_{\ell_{\infty}^{n}}$. It makes us suspect that it is probably more convenient to consider the parameters from
\cite{campos1, GUSE} when estimating $H_{\mathbb{R},1200}\left(  2\right)
$, and it is in fact true, as the following table shows:%
\[%
\begin{tabular}
[c]{lll}%
Polynomial $P_{600}$ & Parameters from \cite{campos1, GUSE} & Lower estimate
for $H_{\mathbb{R},1200}\left(  2\right)  $\\
$(P_{3})^{200}$ & $a=1,~b\approx-1.6692$ & $\geq1.422344$\\
$(P_{5})^{120}$ & $a\approx0.19462,~b\approx0.66008,~c\approx0.97833$ &
$\geq1.549722$\\
$(P_{6})^{100}$ & $a=1,~b\approx-2.2654$ & $\geq1.584313$\\
$(P_{8})^{75}$ & $a\approx0.15258,~b\approx0.64697$ & $\geq1.640430$\\
$(P_{10})^{60}$ & $a\approx0.0938,~b\approx-0.5938.$ & $\geq1.651703$%
\end{tabular}
\ \ \ \ \ \ \
\]
However it remains the possibility of finding even better parameters. As a matter of fact, this is possible via an exhaustive computational search (we fixed one of the parameters and varied the other). We thus obtain:
\[%
\begin{tabular}
[c]{lll}%
Polynomial $P_{600}$ & Slightly better parameters & Lower estimates for
$H_{\mathbb{R},1200}\left(  2\right)  $\\
$(P_{3})^{200}$ & $a=1,~b\approx-1.67053$ & $\geq1.422433$\\
$(P_{5})^{120}$ & $a\approx0.19462,~b\approx0.66,~c\approx0.97833$ &
$\geq1.549744$\\
$(P_{6})^{100}$ & $a=1,~b\approx-2.2663$ & $\geq1.584430$\\
$(P_{8})^{75}$ & $a\approx0.15258,~b\approx0.64698$ & $\geq$ $1.640436$\\
$(P_{10})^{60}$ & $a\approx0.0938,~b\approx-0.5934.$ & $\geq1.65362$%
\end{tabular}
\ \ \ \ \ \ \
\]
We note that the new parameter is very close to the parameter of the case $B_{\ell_{\infty}^{n}}$.

\section{Final comments}

\bigskip A natural problem that arise from our calculations is to identify%
\[
\lim\sup_{m}\sqrt[m]{C_{\mathbb{R},m,2m}^{\mathrm{pol}}}%
\]
when restricted to polynomials in $\ell_{2m}^{2}.$ Further problems also arise
naturally, such as the constants for polynomials in general $\ell_{p}^{m}$
spaces and whether is possible or not to face these problems analytically
instead of numerically.

\end{document}